\documentclass{amsart}
\usepackage{comment}

\address{Department of Mathematics,
University of British Columbia,
Vancouver,  B.C., V6T 1Z2  CANADA}
\email{akhtari@math.ubc.ca}
 

\begin{document}

\newtheorem{thm}{Theorem}[section]
\newtheorem{prop}[thm]{Proposition}
\newtheorem{lemma}[thm]{Lemma}
\newtheorem{cor}[thm]{Corollary}
\newtheorem{conj}[thm]{Conjecture}

\begin{center}
\Large{The Diophantine  equation $aX^{4} - bY^{2} = 1$}
\end{center}

\bigskip

\begin{center}
Shabnam Akhtari
\end{center}

\bigskip

Abstract.
As an application of the method of Thue-Siegel, we will resolve a conjecture of Walsh to the effect that the Diophantine equation $aX^{4} - bY^2=1$, for fixed positive integers $a$ and $b$, possesses at most two solutions in positive integers $X$ and  $Y$. Since there are infinitely many pairs $(a,b)$ for which two such solutions exist, this result is sharp.

\section{Introduction}

In a series of papers over nearly forty years, Ljunggren (see e.g. \cite{Lj1}, \cite{Lj2}, \cite{Lj3}, \cite{Lj4} and \cite{Lj5}) derived  remarkable sharp bounds for the number of solutions to various quartic Diophantine equations, particularly those of the shape
\begin{equation} \label{zero}
aX^4 - bY^2 = \pm 1,
\end{equation}
typically via sophisticated application of Skolem's $p$-adic method. More recently, there has been a resurgence of interest in Ljunggren's work; results along these lines are well surveyed in the paper of Walsh \cite{8}. By way of example, using lower bounds for linear forms in logarithms, together with an assortment of elementary arguments, Bennett and Walsh \cite{1}  showed that the equation 
\begin{equation}\label{1.2}
aX^{4} - bY^{2} = 1
\end{equation}
has at most one solution in positive integers $X$ and $Y$, when $a$ is an integral square and $b$ is a positive integer. For general $a$ and $b$, however, there is no absolute upper bound for the number of integral solutions to (\ref{1.2}) available in the literature, unless one makes strong additional assumptions (see e.g.  \cite{1}, \cite{2}, \cite{Che}, \cite{Coh}, \cite{Lj4}, \cite{Lj5} and \cite{wal}). This lies in sharp contrast to the situation for the apparently similar equation
\begin{equation} \label{muh}
aX^4-bY^2=-1
\end{equation}
where Ljunggren \cite{Lj2} was able to bound the number of positive integral solutions by $2$ for arbitrary fixed $a$ and $b$. Moreover, it appears that the techniques employed to treat equation (\ref{muh}) and, in special cases, (\ref{1.2}), do not lead to results for (\ref{1.2}) in general.

It is our goal in this paper to rectify this situation. To be precise, we will prove the following
\begin{thm}\label{main}
Let $a$ and $b$ be positive integers. Then equation (\ref{1.2})
has at most two solutions in positive integers $(X,Y)$.
\end{thm}
This resolves a conjecture of Walsh (see \cite{1}, \cite{2}, \cite{Coh} and \cite{8}), which had been suggested by computations and assorted heuristics. Since there are infinitely many pairs $(a,b)$ for which two such solutions exist (see Section \ref{sec2}), this result is best possible.

To prove this, we will appeal to classical results of Thue \cite{thu38} from the theory of Diophantine approximation, together with modern refinements, particularly those of Evertse \cite{Ev}.  Such an approach, based on Pad\'e approximation to binomial functions, has been used in a number of previous works to explicitly solve Thue inequalities and equations (see e.g. \cite{2}, \cite{Che}, \cite{Heu1}, \cite{wal}, \cite{Wak0}, \cite{Wak1} ) or to bound the number of such solutions (see e.g. \cite{Ben}, \cite{Ev}, \cite{Kre}). We will apply similar techniques to a certain family of quartic inequalities.

\section{An Equivalent Problem} \label{sec2}

Let $a$ denote a non-square positive integer, and $b$ a positive integer for which the quadratic equation
\begin{equation}\label{1.3}
aX^{2} - bY^{2} = 1
\end{equation}
is solvable in positive integers $X$ and $Y$. Let $(v , w)$ be a pair of positive solutions to (\ref{1.3}) so that 
$$
\tau = v\sqrt{a} + w\sqrt{b} > 1,
$$
and $\tau$ is minimal with this property. All solutions in positive integers of (\ref{1.3}) are given by $(v_{2k+1},w_{2k+1})$, where
$$
\tau^{2k+1} = v_{2k+1}\sqrt{a} + w_{2k+1}\sqrt{b} \qquad  (k \geq 0) 
$$
(see \cite{Walk} for a proof). Solving the quartic equation (\ref{1.2}) is thus equivalent to the problem of determining all squares in the sequence $\{v_{2k+1}\}$. One can find a proof of the following result in \cite{Rot}.
\begin{prop}\label{1.1}
If $v_{2k+1}$ is a square for some $k \geq 0$, then $v_{1}$ is also a square.
\end{prop}

Let us assume that equation (\ref{1.2}) is solvable. Proposition \ref{1.1} implies that $\tau = \tau(a , b)$ is of the form $\tau = x^{2}\sqrt{a} + w\sqrt{b}$. We have
$$\tau = \sqrt{t+1} + \sqrt{t} ,$$
where $t = ax^{4} - 1$. Thus, for $k \geq 0$ 
$$\tau^{2k+1} = V_{2k+1} \sqrt{t+1} + W_{2k+1}\sqrt{t} , $$
where $V_{2k+1} = \frac{v_{2k+1}}{v_{1}}$. Hence, by Proposition \ref{1.1}, $v_{2k+1}$ is a square if and only if $V_{2k+1}$ is a square.  In other words, in order to bound the number of positive integer solutions to an equation of the form $aX^{4} - bY^{2} = 1$, it is sufficient to determine an upper bound for the number of integer solutions to Diophantine equations of the shape
\begin{equation}\label{1.4}
(t+1)X^{4} - tY^{2} = 1.
\end{equation}
The main result of \cite{2} is the following.
\begin{prop}\label{btw}
Let $m$ be a positive integer. Then the only positive integral solutions to the equation 
$$(m^{2}+m+1)X^{4} - (m^{2}+m)Y^{2} = 1$$
are given by $(X , Y) = (1 , 1)$ and $(X , Y) = (2m+1, 4m^{2}+4m+3)$.
\end{prop}
In fact, these are the only values of $t$ for which equation (\ref{1.4}) is known to have as many as two positive solutions (suggesting a stronger version of Theorem \ref{main}).
Note that if $V_{3} = z^{2}$, where $z$ is a positive integer, then since $V_{3} = 1+4t$, we have
$$
4t= z^{2} - 1 = (z - 1) (z+ 1) 
$$
and therefore there exist positive integers $m$ and $n$ such that $t = mn$, $2m=z-1$ and $2n=z+1$. We conclude, therefore, that $n= m+1$ and $t= m^{2} +m$.  Proposition \ref{btw} thus implies the following.
\begin{cor}\label{me}
If $V_{3}$ is a square then for any $k >1$, $V_{2k+1}$ is not a square and there are  only $2$ solutions to equation (\ref{1.4}) in positive integers $X$ and $Y$.
\end{cor}

As it transpires, we will need to account for the possibility of $V_{2k+1}$ being square, for odd values of $k$. The preceding result handles the case $k=1$. For $k=3$ and $k=5$, we will appeal to
\begin{lemma}\label{711}
If $t > 204$, then neither $V_7$ nor $V_{11}$ is an integral square.
\end{lemma}
\begin{proof}
The equation $z^2 = V_{7}=64t^3+80t^2+24t+1$ was treated independently using the function faintp on SIMATH and IntegralPoints on MAGMA, and found to have only the solutions corresponding to $t=0$ and $t=1$ .
For the case $z^2=V_{11}$, we first put $x=4t$, and see that the desired result will follow by determining the set of rational points
on the curve $z^2=x^5+9x^4+28x^3+35x^2+15x+1$. The proof now follows exactly as the proof for the case $M^2=U_{11}$ on pages $8-10$ of \cite{BT}, but with $x=P^2$ and $Q=-1$, as the proof therein does not take into account the fact that $P^2$ is a square..
\end{proof}

\section{Reduction To A Family Of Thue Equations} \label{sec3}

We will begin by applying an argument of Togbe, Voutier and Walsh \cite{wal} to reduce (\ref{1.4}) to a family of Thue equations. We subsequently apply the method of Thue-Siegel to find an upper bound for the number of solutions to this family. Let 
$$
P(x , y) = x^{4} + 4tx^{3}y -6tx^{2}y^{2} - 4t^{2}xy^{3} + t^{2}y^{4} .
$$
The following is a modified version of Proposition $2.1$ of \cite{wal}. We will include a proof primarily for completeness (and since we will have need of one of the inequalities derived therein).
\begin{prop}\label{tvw}
Let $t$ be a positive integer such that $t \neq m^2+m$ for all $m \in \mathbb{Z}$. If $(X , Y) \neq (1,1)$ is a positive integer solution to equation (\ref{1.4}), then there is a solution in coprime positive integers $(x , y)$  to the equation 
\begin{equation*}
 P(x , y) = t_{1}^{2} ,
\end{equation*} 
where $t_{1}$ divides $t$, $t_{1} \leq \sqrt{t}$ and $xy > 64t^3$.
\end{prop}
\begin{proof}
For $k \geq 0$, let us define $\tau$, $V_{2k+1}$ and $W_{2k+1}$ as in Section \ref{sec2}, and choose $T_k$ and $U_k$ to satisfy 
$$
\tau^{2k} = T_{k} + U_{k}\sqrt{t(t+1)}.
$$
Assume that $V_{2k+1} = z^{2}$ for some integer $z>1$. We will suppose that $k$ is odd, $k= 2n+1$ say, as the case that $k$ is even is similar and discussed in \cite{wal}.  When $k=2n+1$,
$$
V_{4n+3} = z^{2} = V_{2n+2}^{2}+V_{2n+1}^{2} = tU_{n+1}^{2} + V_{2n+1}^{2},
$$
with, via Corollary \ref{me}, $n>0$. Thus 
$$
tU_{n}^{2} = z^{2} - V_{2n+1}^{2}=tU_{n}^{2}=z^{2} - (T_{n}+tU_{n})^{2}.
$$
Since $U_{n+1} = 2T_{n}+(2t+1)U_{n}$ and $\gcd(U_{n} , T_{n}) = 1$, we have
$$
\gcd(U_{n+1},T_{n}+tU_{n}) = 1
$$
and hence there exist positive integers $G$, $H$, $t_{1}$, $t_{2}$, with $U_{n+1}=2GH$ and $t=t_{1}t_{2}$, such that
$$
z-(T_{n}+tU_{n}) = 2t_{1}G^{2} \; \mbox{ and } \; z+(T_{n}+tU_{n}) = 2t_{2}H^{2}.
$$
Therefore, $T_{n}+tU_{n} = t_{2}H^{2} - t_{1}G^{2}$, and since 
$$
2GH = U_{n+1} = 2T_{n}+(2t+1)U_{n},
$$
we deduce that 
$$
U_{n} = 2GH - 2t_{2}H^{2} + 2t_{1}G^{2}
$$
and
$$
T_{n}=t_{2}H^{2} -t_{1}G^{2} -t(2GH - 2t_{2}H^{2} +2t_{1}G^{2}).
$$
Substituting for $T_{n}$ and $U_{n}$ in the equation $T_{n}^{2} - t(t+1)U_{n}^{2} = 1$,  we obtain the equation
$$
t_{1}^{2}G^{4} - 4tt_{1}G^{3}H -6tG^{2}H^{2} + 4tt_{2}GH^{3}+ t_2^{2}H^{4} = 1.
$$
Multiplying both sides by $t_{1}^{2}$ and taking $x = -t_{1}G, \, y =H$, we find that $x$ and $y$ are coprime positive integers satisfying $P(x, y) = t_{1}^{2}$. To complete the proof, we observe that, since Lemma \ref{711} and Corollary \ref{me} imply that $n \geq 3$,
\begin{equation}\label{g}
x y = t_{1}GH = \frac{t_{1}}{2} \, U_{n+1} \geq \frac{t_{1}}{2} \, U_{4}  > 64 t^3 .
\end{equation}
\end{proof}

Our focus for the remainder of the paper will be to find, for fixed $t$, an upper bound upon the number of coprime positive integral solutions to the constrained inequality 
\begin{equation}\label{2.1}
0< P(x , y) \leq t^2, \; \; \; xy > 64t^3.
\end{equation}
We should note that for $t \leq 204$, Theorem \ref{main} with $(a,b)=(t+1,t)$ has been verified in \cite{wal}. Here and henceforth, therefore, we will assume that $t > 204$.
To proceed, let $\xi=\xi (x,y)$ and $\eta=\eta (x,y)$ be linear functions of $(x , y)$ so that
$$
\xi^{4} = 4 \, (\sqrt{-t}+1)(x - \sqrt{-t}y)^{4} \    \; \mbox{ and } \;   \    \eta^{4} = 4 \, (\sqrt{-t}-1 )(x + \sqrt{-t}y)^{4}.
$$
We call $(\xi , \eta)$, a pair of \emph{resolvent forms}.
Note that 
$$
P(x , y) = \frac{1}{8}(\xi^{4} - \eta^{4}) 
$$
and if $(\xi , \eta)$ is a pair of resolvent forms then there are precisely three others with distinct ratios, say 
 $(-\xi , \eta)$, $(i \xi , \eta)$ and $(-i\xi , \eta)$. 
Let $\omega$ be a fourth root of unity, $(\xi , \eta)$ a fixed pair of resolvent forms and set
$$
z = 1 - \left(\frac{\eta (x,y)}{\xi (x,y)} \right)^{4}.
$$
We say that the integer pair $(x , y)$ is \emph{related} to $\omega$ if 
$$
 \left|\omega- \frac{\eta (x,y)}{\xi (x,y)}  \right| < \frac{\pi}{12} |z|.
$$
It turns out that each nontrivial solution $(x,y)$ to (\ref{2.1}) is related to a fourth root of unity :

\begin{lemma} \label{lipp}
Suppose that $(x , y)$ is a positive integral solution to inequality (\ref{2.1}), with 
$$
\left|\omega_{j} - \frac{\eta(x , y)}{\xi(x , y)} \right| = \min_{0 \leq k \leq 3} \left|e^{k\pi i/2} - \frac{\eta(x , y)}{\xi(x , y)} \right|.
$$
Then
\begin{equation}\label{Gap2}
|\omega_{j} - \frac{\eta(x,y)}{\xi(x,y)}| < \frac{\pi}{12} |z(x , y)|.
\end{equation}
\end{lemma}
\begin{proof}
We begin by noting that 
$$
|z| = \left|\frac{\xi^{4} - \eta^{4}}{\xi^{4}} \right| = \frac{8 \, P(x , y)}{|\xi^{4}|},
$$
and, from $xy \neq 0$, 
$$
|\xi^{4}(x , y)| \geq 4 (\sqrt{1+t})^{5},
$$
whereby
$$
|z|  \leq \frac{2 t^2}{(\sqrt{t+1})^{5}} < 1.
$$
Since $\eta = - \bar{\xi}$, it follows that 
$$
\left|\frac{\eta}{\xi} \right| = 1, \; \;  |1 - z| = 1.
$$
Now let
$4  \theta = \textrm{arg} \left( \frac{\eta(x , y)^{4}}{\xi(x , y)^{4}}\right)$.
We have
 $$
 \sqrt{2 - 2\cos(4\theta)} = |z|  < 1,
 $$
 and so $|\theta| < \frac{\pi}{12}.$ Since 
$$
\left |\omega_{j} - \frac{\eta(x,y)}{\xi(x,y)} \right| \leq  |\theta|, 
$$
it follows that
$$
\left|\omega_{j} - \frac{\eta(x,y)}{\xi(x,y)} \right| \leq \frac{1}{4} \frac{|4\theta|}{\sqrt{2 - 2\cos(4\theta)}} 
\left|1 - \frac{\eta(x,y)^{4}}{\xi(x,y)^{4}} \right|. 
$$
From the fact that $\frac{|4\theta|}{\sqrt{2 - 2\cos(4\theta)}} < \frac{\pi}{3}$ whenever $ 0<|\theta| < \frac{\pi}{12},$ we obtain inequality (\ref{Gap2}), as desired.
\end{proof}
 
 This lemma shows that each integer pair $(x , y)$ is related to precisely one fourth root of unity. Let us fix such a fourth root, say $\omega$, and suppose that we have distinct coprime positive solutions $(x_1,y_1)$ and $(x_2,y_2)$ to inequality (\ref{2.1}), each related to $\omega$. We will assume, as we may, that  $|\xi(x_{2} , y_{2})| \geq |\xi(x_{1} , y_{1})|$. For concision, we will write $\eta_{i} = \eta(x_{i} , y_{i})$ and $\xi_{i} = \xi(x_{i} , y_{i})$. Before we move into the heart of our proof, we will mention a pair of results that will be the starting point for our later proving that $(x_1,y_1)$ and $(x_2,y_2)$ are far apart in height.
 
Since 
\begin{equation} \label{frog}
|z| = \frac{8 P(x,y)}{|\xi|^{4}} \leq \frac{8 t^2}{|\xi|^{4}},
\end{equation}
it follows from (\ref{Gap2}) that 
 \begin{equation} \label{fish}
 |\xi_{1} \eta_{2} - \xi_{2} \eta_{1}| = |\xi_{1}(  \eta_{2} - \omega \xi_{2}) - \xi_{2}( \eta_{1} - \omega \xi_{1})| \leq \frac{2\pi}{3}t^2 \, \left(\frac{|\xi_{1}|}{|\xi_{2}^{3}|} + \frac{|\xi_{2}|}{|\xi_{1}^{3}|} \right) \leq  \frac{ 4\pi t^2 \, |\xi_{2}|}{3|\xi_{1}^{3}|}.
\end{equation}
On the other hand, choosing our fourth root appropriately, we have
$$ 
\left( \begin{array}{cc}
   \sqrt{2} (\sqrt{-t}+1)^{1/4} & - \sqrt{2} (\sqrt{-t}+1)^{1/4} \sqrt{-t}\\
   \sqrt{2} (\sqrt{-t}-1)^{1/4} &  \sqrt{2} (\sqrt{-t}-1)^{1/4}\sqrt{-t}
    \end{array} \right) 
 \left( \begin{array}{cc}
 x_{1} & x_{2}\\
 y_{1} & y_{2}
  \end{array} \right)  = 
    \left( \begin{array}{cc}
\xi_{1} & \xi_{2}\\
 \eta_{1} & \eta_{2}
  \end{array} \right)
$$
and so
$$ 
|\xi_{1} \eta_{2} - \xi_{2} \eta_{1} | = \left| 4 (t+1)^{1/4}\sqrt{t} \, (x_{1}y _{2}- x_{2}y_{1}) \right|.
$$
Since   $x_{1}y _{2}- x_{2}y_{1}$ is a nonzero integer (recall that we assumed $\gcd (x_i,y_i)=1$), we have
 \begin{equation}\label{lb}
 |\xi_{1} \eta_{2} - \xi_{2} \eta_{1} |  \geq 4 \sqrt{t} \, (t+1)^{1/4}
 \end{equation}
and thus, combining (\ref{fish}) and (\ref{lb}), we conclude that
if $(x_{1} , y_{1})$ and $(x_{2} , y_{2})$ are distinct solutions to (\ref{2.1}), related to $\omega$, with
$|\xi(x_{2} , y_{2})| \geq |\xi(x_{1} , y_{1})|$
then 
\begin{equation}\label{Gap}
|\xi_{2}| > \frac{3}{\pi}t^{-5/4} \, |\xi_{1}|^{3}.
\end{equation}

As a final preliminary result, we have the following lemma, whose proof is an immediate consequence of the definition of resolvent forms :

\begin{lemma}\label{ai}
If $(x_{1} , y_{1})$ and $(x_{2}, y_{2})$ are two pairs of rational integers then
$$
\frac{\xi(x_{1} , y_{1}) \eta(x_{2} , y_{2})}{(-t-1)^{1/4}}, \; \;  \xi(x_{1} , y_{1})^{3} \xi(x_{2} , y_{2}) \;  \mbox{ and } \; \eta(x_{1} , y_{1})^{3} \eta(x_{2} , y_{2}) 
$$
are integers in $\mathbb{Q}(\sqrt{-t})$.
\end{lemma}

\section{ Pad\'e Approximation}

The main focus of this section is to construct  a family of dense approximations to $\xi/\eta$ from rational function approximations to the binomial function $(1 - z)^{1/4}$. Consider the system of linear forms 
$$
R_{r}(z) = -Q_{r}(z) + (1-z)^{1/4}P_{r}(z),
$$ 
where $R_{r}(z) = z^{2r+1}\bar{R}_{r}(z)$, $\bar{R}_{r}(z)$ is regular at $z=0$, and $P_{r}(z)$ and $Q_{r}(z)$ are polynomials of degree $r$. Thue \cite{thu34}, \cite{thu35} explicitly found polynomials $P_{r}(z)$ and $Q_{r}(z)$  that satisfy such a relationship, and  Siegel \cite{Sie28} identified them in terms of hypergeometric  polynomials. Refining the work of Thue and Siegel,  Evertse \cite{Ev} used the theory of hypergeometric functions to sharpen Siegel's upper bound for the number of solutions to the equation $f(x , y) = 1$, where $f$ is a cubic binary form with positive discriminant. In this paper, we will apply similar arguments to certain quartic forms.

We begin with some preliminaries on hypergeometric functions. A {\em hypergeometric function} is a power series of the shape
$$ 
F(\alpha , \beta , \gamma , z) =   1 + \sum_{n =1}^{\infty} \frac{\alpha (\alpha + 1) \cdots (\alpha + n  -1) \beta (\beta + 1) \cdots (\beta + n - 1)}{\gamma (\gamma + 1) \cdots (\gamma + n - 1) n!} z^{n}.
$$
Here $z$ is a complex variable and $\alpha$, $\beta$  and $\gamma$ are complex constants. If $\alpha$ or $\beta$ is a non-positive integer  and $m$ is the smallest integer such that 
$$
\alpha    (\alpha + 1) \cdots (\alpha + m) \beta (\beta + 1) \cdots (\beta + m) = 0,
$$
then $F(\alpha , \beta , \gamma , z) $ is a polynomial in $z$ of degree $m$. Furthermore, if $\gamma$  is a non-positive integer, we will assume that at least one of $\alpha$ and $\beta$ is also a non-positive integer, smaller than $\gamma$.

We note that $F( \alpha , \beta , \gamma , z)$ converges for $|z| < 1$.  By a result of Gauss, if $\alpha$, $\beta$  and $\gamma$ are real with $\gamma > \alpha + \beta $ and $\gamma$, $ \gamma - \alpha$ and $\gamma - \beta$ are not non-positive integers, then $F(\alpha , \beta , \gamma , z)$ converges for $z = 1$ and we have
\begin{equation}\label{Gam}
F(\alpha , \beta , \gamma , 1) = \frac{\Gamma(\gamma) \Gamma(\gamma - \alpha - \beta)}{\Gamma(\gamma - \alpha) \Gamma(\gamma - \beta)}.
\end{equation} 
For future use, it is worth noting that the hypergeometric function $F( \alpha , \beta , \gamma , z)$ satisfies the differential equation 
 \begin{equation} \label{hip}
 z (1 - z) \frac{d^{2}F}{dz^{2}} + (\gamma - (1 + \alpha + \beta ) z )\frac{dF}{dz} - \alpha \beta F = 0.
 \end{equation}
 
 Our family of dense approximations to $\xi/\eta$ are as given in the following lemma; their connection to hypergeometric functions will be made apparent later.
\begin{lemma}\label{hyp}
Let $r$ be a positive integer and $g \in \{ 0 , 1 \}$. Put
\begin{eqnarray}\label{AB}\nonumber
A_{r, g} (z) & = &  \sum_{m =0}^{r}{r - g + \frac{1}{4} \choose m} {2r - g - m \choose r - g}   (-z)^{m},   \\  
B_{r, g} (z) & = & \sum_{m =0}^{r-g}{r - \frac{1}{4} \choose m} {2r - g - m \choose r }   (-z)^{m}.
 \end{eqnarray}
  \flushleft
 \begin{itemize}
 \item[(i)] There exists a power series $F_{r,g}(z)$ such that for all complex numbers $z$ with $|z| < 1$
 \begin{equation}\label{ABF}
 A_{r,g}(z) - (1 - z)^{1/4} B_{r, g}(z) = z^{2r+1 -g}F_{r,g}(z)
 \end{equation}
  and 
  \begin{equation}\label{F}
 |F_{r,g}(z)| \leq \frac{{r-g+1/4 \choose r+1-g} {r- 1/4 \choose r}}{{2r + 1 - g \choose r}} (1 - |z|)^{-\frac{1}{2}(2r + 1 - g)}.
 \end{equation}
 \item[(ii)] For all complex numbers $z$ with $|1 - z| \leq 1$ we have 
 \begin{equation}\label{A}
 |A_{r,g}(z)| \leq {2r - g \choose r}.
 \end{equation}
 
 \item[(iii)] For all complex numbers $z \neq 0$ and for $h \in \{1 , 0\}$ we have
 \begin{equation}\label{BA}
 A_{r, 0}(z) B_{r+h , 1 , 1}(z) \neq A_{r+h , 1}(z) B_{r , 0}(z).
 \end{equation} 
 \end{itemize}
\end{lemma}

\begin{proof}
Put
$$
C_{r,g}(z) = \sum_{m =0}^{r} {r - 1/4 \choose r - m }{r - g + 1/4 \choose m} z^{m} 
$$
and 
$$
D_{r,g} = \sum_{m = 0}^{r - g} {r - 1/4 \choose m} {r - g + 1/4 \choose r - g + m} z^{m}.
$$
Note that, in terms of hypergeometric functions,
$$
A_{r, g}(z) = {2r - g \choose r} F (-1/4 - r + g , -r , -2r + g , z),
$$ 
$$
B_{r, g}(z) = {2r - g \choose r - g} F (1/4 - r  , -r + g , -2r + g , z), 
$$ 
$$
C_{r, g}(z) = {r - 1/4 \choose r} F (-1/4 - r + g , -r , 3/4, z)
$$
and 
$$
D_{r, g}(z) = {r - g + 1/4 \choose r- g} F (1/4- r  ,  -r + g , 5/4,z), 
$$ 
We will begin by proving that
$$
C_{r,g}(z) = A_{r,g}(1 - z)  , \  D_{r , g}(z) = B_{r,g}(1 - z).
$$
The power series $ F(z) = \sum_{m =0}^{\infty} a_{m} z^{m}$ is a solution to the differential equation (\ref{hip}) 
precisely when
\begin{equation}\label{an}
(n + 1) (\gamma + n) a_{n + 1}  = (\alpha + n) ( \beta + n ) a_{n}   \  \textrm{for}  \  n = 0, 1 , 2 , \ldots  .
\end{equation}
Both $A_{r,g}(1 - z) $ and $C_{r,g}(z)$ satisfy  (\ref{hip}) with $\alpha = -1/4 - r + g$ , $\beta = -r$ , $\gamma = 3/4$.  Since $\gamma$ is not a non-positive integer, all coefficient $a_{i}$ of power series $y(z)$ are determined by $a_{0}$. Hence the solution space of (\ref{hip}) is one-dimensional. Therefore, $A_{r,g}(1 - z)$ and $C_{r,g}(z)$ are linearly dependent.
On equating the coefficients of $z^{r}$ in 
$$
(1 + z) ^{2r + g} = (1 + z)^{r - 1/4} (1 + z)^{r - g + 1/4}, 
$$
we find that
\begin{displaymath}\nonumber
C_{r,g} (1) =  \sum_{m = 0}^{r} {r - 1/4 \choose r -m }{r - g + 1/4 \choose m} = {2r - g \choose r} = A_{r,g}(0),
\end{displaymath}
and hence $C_{r,g}(z) = A_{r,g}(1 - z)$. Similarly, $D_{r,g}(z) = B_{r,g}(1 - z)$.
One can easily observe that $C_{r, g}(z)$  has positive coefficients. Hence when $|1 - z| \leq 1$,
$$
|A_{r,g}(z)| = |C_{r,g}(1 - z)| \leq C_{r,g}(1) = A_{r,g}(0) = {2r - g \choose r}.
$$
This proves part (ii) of our lemma.

To prove (\ref{ABF}), we define
  $$
  G_{r,g}(z) = F(r+1-g , r+3/4 , 2r +2 -g , z)
  $$ 
 and notice that, for $|z| < 1$, the functions $A_{r,g}(z)$, $(1 - z)^{1/4} B_{r,g}(z)$ and $z^{2r+1 -g}G_{r,g}(z)$ satisfy (\ref{hip}) with $\alpha = -1/4 - r + g$ , $\beta = -r$ , $\gamma = -2r + g$.
  Suppose
  \begin{displaymath} 
  G_{r,g}(z) = \sum_{m =0}^{\infty} g_{m} z^{m}.
  \end{displaymath}
  We have $g_{0} = 1$ and, for $m \geq 0$,
 $$
  \frac{g_{m+1}}{g_{m}}  =   \frac{(r + 1- g + m)(r + 3/4 +m)}{(m + 1)(2r + 2 - g + m)} \leq  \frac{r + 1/2 - g/2 + m}{m + 1} = \frac {(-1)^{m+1} {-r -1/2 + g/2 \choose m+1 }}{(-1)^{m}{-r - 1/2 + g/2 \choose m}}.
$$
Therefore,
  $$
  |G_{r,g}(z)| \leq \sum_{m = 0}^{r} {-r -1/2 + g/2 \choose m} (-|z|)^{m} = (1 - |z|)^{-\frac{1}{2}(2r+1-g)}.
  $$
 Since $r \geq 1$ and $g \in \{0 , 1 \}$, $\gamma = -2r + g$ is a negative integer. By (\ref{an}), If $F(z) = \sum_{m= 0}^{\infty} a_{m} z^{m}$ is a solution to (\ref{hip}), then since $a_{0}$ and $a_{2r - g+1}$  may vary independently, the solution space of (\ref{hip}) is two-dimensional. Therefore, there are constants $c_{1}$, $c_{2}$ and $c_{3}$, not all zero, such that 
 $$
 c_{1}A_{r,g}(z) + c_{2} (1 - z) ^{1/4}B_{r,g}(z) + c_{3} z^{2r + 1 - g} G_{r,g}(z) = 0.
 $$
 Letting $z = 0$, since $A_{r,g}(0) = B_{r,g}(0) \neq 0 $, we find that $c_{1} = - c_{2} \neq 0$. We may thus assume $c_{1} = 1$. Substituting $z=1$ in above identity thus yields $c_{3} = -\frac{A_{r,g}(z)}{G_{r,g}(z)}$, whence
 $$
 F_{r,g}(z) = A_{r,g}(1) G_{r,g}(1)^{-1} G_{r,g}(z).
 $$ 
 In order to complete the proof of part (i), note that,  by (\ref{Gam}), we have
 $$
 A_{r,g}(1) G_{r,g}(1)^{-1}  = {r - 1/4 \choose r} \frac{\Gamma(r +1) \Gamma(r + 5/4 -g)}{\Gamma(2r + 2 -g) \Gamma(1/4)} =\frac{ {r - 1/4 \choose r} {r - g + 1/4 \choose r + 1 - g}}{{2r+1-g \choose r}} .
$$
  
    It remains to prove part (iii). By (\ref{ABF}),
   $$
   A_{r,0}(z) B_{r+h , 1}(z) - A_{r+h , 1}(z)B_{r,0}(z) = z^{2r+h}P_{r,h}(z),
   $$
   where $P_{r,h}(z)$ is a power series. However, the left hand side of the above identity is a polynomial of degree at most $2r + h$, and so $P_{r,h}$ must be a constant.  Letting $z = 1$, we obtain that $P_{r,h}$ is not $0$. Therefore, $$A_{r,0}(z) B_{r+h , 1}(z) - A_{r+h , 1}(z)B_{r,0}(z) = 0$$ if and only if $z = 0$.
 \end{proof}

\section{Some Algebraic Numbers}

Combining our polynomials of the previous section with the resolvent forms defined in Section \ref{sec3}, we will  consider the complex sequences $\Sigma_{r,g} $ given by
 $$
 \Sigma_{r,g} = \frac{\eta_{2}}{\xi_{2}}A_{r,g}(z_{1}) - (-1)^{r}\frac{\eta_{1}}{\xi_{1}} B_{r,g}(z_{1})
 $$
 where $z_{1} = 1 - \eta_{1}^{4}/ \xi_{1}^{4}$.
Define
$$
\Lambda_{r,g} =\frac{ \xi_{1}^{4r+1-g}\xi_{2}}{(-t-1)^{1/4}} \Sigma_{r,g}.
$$
We will show that  $\Lambda_{r,g}$ is either an integer in $\mathbb{Q}(\sqrt{-t})$ or a fourth root of such an integer.  If $\Lambda_{r,g} \neq 0$, this provides a lower bound upon 
 $|\Lambda_{r,g}|$. In conjunction with the inequalities derived in Lemma \ref{hyp}, this will induce a strong ``gap principle'', guaranteeing that solutions to inequality (\ref{2.1}) must, in a certain sense, increase rapidly in height.
  
For a polynomial $P(z)$ of degree $n$, we will denote by $P^{*}(x , y) = x^{n} P(y/x)$ an associated binary form.
Let $A_{r,g}$ and $B_{r,g}$ be as in (\ref{AB}) and, as in the proof of Lemma \ref{hyp}, set
 $$
 C_{r,g}(z) = A_{r,g}(1 - z)  \; \mbox{ and } \;  D_{r , g}(z) = B_{r,g}(1 - z).
 $$
 For $z \neq 0$, we have $D_{r , 0}(z) = z^{r} C_{r,0}(z^{-1})$, hence
\begin{eqnarray}\label{star} \nonumber 
A^{*}_{r , 0}\left(z , z + \bar{z} \right) 
& = & z^{r} A_{r , 0} \left(1 + \frac{\bar{z}}{z} \right) = z^{r} C_{r,0} \left(\frac{-\bar{z}}{z} \right) \\   
& = & (-1)^{r}\bar{z}^{r} D_{r,0} \left(\frac{-z}{\bar{z}} \right) = (-1)^{r}\bar{z}^{r} B_{r,0} \left(1 +\frac{z}{\bar{z}} \right) \\  \nonumber
& = & (-1)^{r}B^{*}_{r , 0} \left(\bar{z} , \bar{z} +z \right) =(-1)^{r} \bar{B}^{*}_{r,0} \left(z , z +\bar{z} \right) .
\end{eqnarray}

\begin{lemma}\label{24}
For any pair of integers $(x , y)$, both $A^{*}_{r,g}(\xi^{4}(x , y), \xi^{4}(x , y) -\eta^{4}(x , y))$  and $B^{*}_{r,g}(\xi^{4}(x , y), \xi^{4}(x , y) -\eta^{4}(x , y) )$ are algebraic integers in 
$\mathbb{Q}(\sqrt{-t})$.
\end{lemma}

\begin{proof}

It is clear that $A^{*}_{r,g}(\xi^{4}(x , y), \xi^{4}(x , y) -\eta^{4}(x , y))$  and $B^{*}_{r,g}(\xi^{4}(x , y), \xi^{4}(x , y) -\eta^{4}(x , y) )$ belong to  $\mathbb{Q}(\sqrt{-t})$; we need only show that they are algebraic integers. From the definitions of $A^{*}_{r,g}(x,y)$, $B^{*}_{r,g}(x,y)$, $\xi (x,y)$ and $\eta(x,y)$ (in particular, since $\xi^{4}(x , y) -\eta^{4}(x , y)= 8P(x , y)$), this is an immediate consequence of Lemma 4.1 of \cite{Chu}, which, in this case, implies that 
$$
\binom{a/4}{n} 8^{n}
$$
is, for fixed nonnegative integers $a$ and $n$, a rational integer.
\end{proof}

We now proceed to show that $\Lambda_{r,g}$ has the desired property. We have
$$
 \Lambda_{r,g} =\frac{\xi_{1}^{1 -g} \eta_{2}}{(-t-1)^{1/4}}A^{*}_{r,g}(\xi^{4} , \xi^{4} - \eta^{4}) - \frac{(-1)^{r}\xi_{1}^{2g}\xi_{2} \eta_{1}}{(-t-1)^{1/4}} B^{*}_{r,g}(\xi^{4} , \xi^{4} - \eta^{4}).
 $$
By Lemmas \ref{ai}, \ref{24} and (\ref{star}), $\Lambda_{r , 0} \in \mathbb{Z}  \sqrt{-t}$. Similarly,  Lemmas \ref{ai} and \ref{24} imply that $\Lambda_{r , 1}^{4}$ is an algebraic integer in $\mathbb{Q}(\sqrt{-t})$. We claim that it is not a rational integer. To see this, let us start by noting that
$$
\frac{\Sigma_{r,g}}{(-t-1)^{1/4}} = \frac{\eta_{2}}{\xi_{2}}A_{r,g}(z_{1}) -(-1)^{r} \frac{\eta_{1}}{\xi_{1}} B_{r,g}(z_{1}) = \frac{\eta}{\xi} \big{(}\frac{\eta_{2}/\eta}{\xi_{2}/\xi}A_{r,g}(z_{1}) - (-1)^{r}\frac{\eta_{1}/\eta}{\xi_{1}/\xi} B_{r,g}(z_{1})\big{)},
$$
where $\eta = (\sqrt{-t}-1)^{1/4}$ and $\xi = (\sqrt{-t}+1)^{1/4}$.
By Lemma \ref{24},   
$$
\frac{\eta_{2}/\eta}{\xi_{2}/\xi}A_{r,g}(z_{1}) - (-1)^{r}\frac{\eta_{1}/\eta}{\xi_{1}/\xi} B_{r,g}(z_{1})\in \mathbb{Q}(\sqrt{-t})
$$
and so
\begin{eqnarray}
\mathfrak{f} = \mathbb{Q}(\sqrt{-t} , \Sigma_{r,g})  & =& \mathbb{Q}(\sqrt{-t} , (-t-1)^{1/4} \frac{\eta}{\xi}) \\ \nonumber
& = & \mathbb{Q}(\sqrt{-t} , (-t+1-2\sqrt{-t})^{1/4}).
\end{eqnarray}
If we choose a complex number $X$ so that $\xi (X , 1) = \eta(X , 1)$ then $X  \in  \mathfrak{f}$ and
$$
P (X , 1) = \frac{1}{8}(\xi^{4}( X , 1) - \eta^{4}(X , 1) )= 0.
$$
Since we have assumed that $P$ is irreducible,  $X$ and $\Sigma_{r , g}$ both have degree $4$ over $\mathbb{Q}(\sqrt{-t})$.

 Suppose that $\Lambda_{r,1}^{4} \in \mathbb{Z}$. Then we have for some  $\rho, \rho_{1}  \in \{ \pm1 , \pm i \}$, that 
 $\Lambda_{r,1} = \rho \bar{\Lambda}_{r,1}$ and $(-t-1)^{1/4} = \rho_{1} \overline{(-t-1)^{1/4}})$, whence, from Lemma \ref{ai},
  \begin{eqnarray}\nonumber
  \Sigma_{r , 1} &= & (-t-1)^{1/4}\xi_{1}^{-4r} \xi_{2}^{-1} \rho \bar{\Lambda}_{r,1}  \\ \nonumber
& =& \xi_{1}^{-4r} \xi_{2}^{-1} \eta_{1}^{4r}\eta_{2}\rho \rho_{1}\left( \frac{\xi_{2}}{\eta_{2}} A_{r,1} \left(1 - \frac{\xi^{4}}{\eta^{4}} \right) - (-1)^{r}\frac{\xi_{1}}{\eta_{1}}B_{r , 1} \left( 1 - \frac{\xi^{4}}{\eta^{4}} \right) \right)\\  \nonumber
  &=& \rho\rho_{1} \frac{\eta_{1}^{4r}}{\xi_{1}^{4r}} \left(A_{r,1}\left( 1- \frac{\xi_{1}^{4}}{\eta_{1}^{4}}\right) - (-1)^{r}\frac{\xi_{1} \eta_{2}}{\xi_{2}\eta_{1}} B_{r ,1} \left( 1- \frac{\xi_{1}^{4}}{\eta_{1}^{4}}\right) \right) .
    \end{eqnarray}
    This together with Lemmas \ref{ai} and  \ref{24} imply that $\Sigma_{r,1} \in \mathbb{Q}(\sqrt{-t} , \rho \rho_{1})$, which contradicts the fact that $\Sigma_{r,1}$ has degree $4$ over $\mathbb{Q}(\sqrt{-t})$. We conclude that $\Lambda_{r,1}$ can not be a rational integer. 

From the well-known characterization of algebraic integers in quadratic fields, we may therefore conclude that,
if $\Lambda_{r,g} \neq 0$, $g \in \{ 0, 1 \}$, then
  \begin{equation}\label{ub}
 |\Lambda_{r,g}| \geq 2^{\frac{-g}{4}} \; t^{\frac{1}{2} - \frac{3g}{8}}.
 \end{equation}

\section{Three Auxiliary Lemmas} 
 
 We will now combine inequality (\ref{ub}) with upper bounds from Lemma \ref{hyp} to show that solutions to (\ref{2.1}) are widely spaced :
 
 \begin{lemma}\label{c}
If $\Sigma_{r,g} \neq 0$, then
$$
c_{1}(r,g)  \, |\xi_{1}|^{4r+1-g}|\xi_{2}|^{-3}+ c_{2}(r,g) \, |\xi_{1}|^{-4r-3(1-g)}|\xi_{2}| > 1,
$$ 
where we may take
$$
c_{1}(r,g) =  \frac{2^{2r+1+g/4}}{\sqrt{\pi r}} \, t^{5/4+3g/8} 
$$
and
$$
c_{2}(r,g) = \frac{2^{1/2+g/4- 2r} 3^{4r+2-2g}}{\pi \sqrt{r}} t^{4r+5/4-13g/8}.
$$
\end{lemma}
  
\begin{proof}
By (\ref{ABF}), we can write
$$
\left| (t+1)^{1/4} \, \Lambda_{r,g} \right|=  |\xi_{1}|^{4r+1-g}  |\xi_{2}|  \   \left| \left(\frac{\eta_{2}}{\xi_{2}} - \omega \right)A_{r,g}(z_{1}) + \omega z_{1}^{2r+1-g}F_{r,g}(z_{1}) \right|. 
$$
Since $|1 - z_{1}| = 1$ and $|z_{1}| \leq 1,$ from (\ref{Gap2}), (\ref{frog}), (\ref{F}),  (\ref{A}), and the inequality 
$$
|\xi_1|^4 > 4 \, (1+t)^{5/2},
$$
we have
$$
\left|(t+1)^{1/4} \, \Lambda_{r,g} \right| \leq  |\xi_{1}|^{4r+1-g}   |\xi_{2}| \left( {2r - g \choose r}   \frac{2 t^2}{|\xi^{4}_{2}|}  + \frac{{r-g+1/4 \choose r+1-g} {r - 1/4 \choose r}}{{2r+1-g \choose r}}               \left( \frac{9t^2}{|\xi^{4}_{1}|}\right)^{2r+1-g} \right) . 
$$
Comparing this with (\ref{ub}), we obtain
$$
c_{1}(r,g) \, |\xi_{1}|^{4r+1-g}|\xi_{2}|^{-3} + c_{2}(r,g) \, |\xi_{1}|^{-4r-3(1-g)}|\xi_{2}| > 1,
$$ 
where we may take $c_{1}$ and $c_{2}$ so that 
$$
c_{1}(r,g) \geq 2^{1+g/4} \, t^{5/4+3g/8} \,   {2r \choose r} 
$$
and
$$
c_{2}(r,g) \geq 2^{g/4} \, 3^{4r+2-2g} \, t^{4r+5/4-13g/8} \, \frac{{r-g+1/4 \choose r+1-g} {r - 1/4 \choose r}}{{2r+1-g \choose r}}.
$$
Applying the following version of Stirling's formula (see Theorem (5.44) of \cite{Str}) 
$$
\frac{1}{2\sqrt{k}}4^{k} \leq {2k \choose k} < \frac{1}{\sqrt{\pi k}}  4^{k},
$$
(valid for $k \in \mathbb{N}$) leads immediately to the stated choice of $c_1$.

To evaluate $c_{2}(r , g)$, we begin by noting that
$$
{2r+1-g \choose r} \geq {2r \choose r} \geq \frac{4^{r}}{2\sqrt{r}}.
$$
Next we will show that
$$
{r-g+1/4 \choose r+1-g} {r - 1/4 \choose r} < \frac{1}{\sqrt{2}\pi r},
$$
for $r \in \mathbb{N}$ and $g \in \{ 0 , 1\}$, whence we may conclude that
$$  \frac{{r-g+1/4 \choose r+1-g} {r - 1/4 \choose r}}{{2r+1-g \choose r}} < \frac{\sqrt{2}}{\sqrt{r}\pi 4^{r}}.
$$
This gives the desired value for $c_{2}(r , g)$. To bound  ${r-g+1/4 \choose r+1-g} {r - 1/4 \choose r}$, first we note that
   $$
   {r - 3/4 \choose r} >{r + 1/4 \choose r+1},
   $$
   for $r \in \mathbb{N}$.
  Put
   $$
   X_{r} = {r - 3/4 \choose r} {r - 1/4 \choose r}  = \frac{y_{r}}{r},
   $$
  whereby
  $$
  X_{r+1} = {r +1/4 \choose r+1} {r + 3/4 \choose r+1}   = \left( \frac{r^{2} + r + 2/9}{r^{2 }+ r} \right)\frac{y_{r}}{r+1}.
  $$
  Hence,
  $$ 
  y_{1} = 3/16   \     ,    \    y_{r} = 3/16 \prod_{k = 1}^{r -1} \frac{k^{2} + k + 3/16}{k^{2} + k}.
  $$
   Since
   $$
   \prod_{k = 1}^{\infty} \frac{k^{2} + k + 3/16}{k^{2} + k}  = \frac{16}{3 \Gamma(1/4) \Gamma(3/4)} =\frac{16}{3\sqrt{2}\pi},  
   $$
  we obtain
    $$
    X_{r} < \frac{1}{\sqrt{2}\pi r},
    $$
  which completes the proof.
   \end{proof}
 
We will also have need of the following :

\begin{lemma}\label{nv}
 If $r \in \mathbb{N}$ and $h \in \{ 0, 1 \}$, then at most one of $\left\{ \Sigma_{r,0}, \Sigma_{r+h,1} \right\}$ can vanish.
  \end{lemma}
\begin{proof}
    Let $r$ be a positive integer and  $h \in \{ 0 , 1\}$ . Following an argument of Bennett \cite{Ben}, we define the matrix $\mathbf{M}$:
   \begin{displaymath}
  \mathbf{M} =
  \left( \begin{array}{ccc}
  A_{r,0}(z_{1}) & A_{r+h , 1}(z_{1}) &(-1)^{r} \frac{\eta_{1}}{\xi_{1}} \\
  A_{r,0}(z_{1}) & A_{r+h , 1}(z_{1}) &(-1)^{r} \frac{\eta_{1}}{\xi_{1}} \\
 B_{r,0}(z_{1}) & B_{r+h , 1}(z_{1}) & \frac{\eta_{2}}{\xi_{2}} 
 \end{array} \right) .
 \end{displaymath}
 The determinant of $\mathbf{M}$ is zero because it has two identical rows. Expanding along the first row, we find that
  $$
  A_{r, 0}(z_{1}) \Sigma_{r+h , 1} -   A_{r+h, 1}(z_{1}) \Sigma_{r , 0} + (-1)^{r}\frac{\eta_{1}}{\xi_{1}}( A_{r, 0}(z_{1}) B_{r+h ,1}(z_{1})  - A_{r+h, 1}(z_{1}) B_{r ,0}(z_{1}) )
  $$ 
  vanishes and hence if $\Sigma _{r,0} = \Sigma _{r + h,1} = 0$,  then 
  $$
  A_{r, 0}(z_{1}) B_{r+h ,1}(z_{1})  - A_{r+h, 1}(z_{1}) B_{r ,0}(z_{1}) = 0,  
  $$
  contradicting part (iii) of Lemma \ref{hyp}.
 \end{proof}
          
 Our final result of this section follows similar lines to an argument of Evertse \cite{Ev}. We show :
 
 \begin{lemma} \label{mush}
 Suppose that $t > 204$. For $r \in \{1 , 2, 3, 4, 5\}$, we have
   $$
   \Sigma_{r,0} \neq 0.
   $$
   \end{lemma}
   \begin{proof}
   
   Let $r \in \{1 , 2,  3, 4, 5 \}$ and suppose that $\Sigma_{r , 0} = 0$. From (\ref{ABF}),  for each
  $r$, the polynomial 
  $$
  A_{r,0}(z)^{4} - (1 -z)B_{r,0}^{4}
  $$
  has a zero at $0$ of order at least $2r+1$. We can thus find polynomials
   $A_r(z), B_r(z)$ and $F_{r}(z) \in \mathbb{Z}[z]$, satisfying 
 $$
 A_{r}(z)^{4} - (1 -z)B_{r}^{4} = z^{2r+1} F_{r}(z).
 $$
 In fact, we have
 $$
 A_{1}(z) = 4 A_{1, 0}(z) = 8 - 5z , 
 $$
 $$
B_{1}(z) = 4 B_{1 , 0}(z) = 8 - 3z , 
$$
$$
F_{1}(z) = 320 - 320z + 81z^{2} ,
$$
 $$
 A_{2}(z) = \frac{32}{3} A_{2,0}(z) = 64 -72z + 15 z^{2}, 
 $$
  $$
  B_{2}(z)= \frac{32}{3} B_{2,0}(z) = 64 -56z + 7 z^{2},
  $$ 
 $$F_{2}(z) = 86016 - 172032z + 114624z^{2} - 28608z^{3} + 2401z^{4},
 $$
 $$
 A_{3}(z)= 128A_{3,0}(z)= 2560 - 4160z + 1872z^{2} -195z^{3},
  $$ 
   $$
  B_{3}(z)= 128B_{3,0}(z)=2560 - 3520z + 1232z^{2} - 77z^{3}, 
  $$ 
  \begin{eqnarray*}
  F_{3}(z) = & & 14057472000 - 42172416000z + 48483635200z^{2} - 26679910400z^{3}\\ 
  & & + 7150266240z^{4}  - 839047040z^{5} + 35153041z^{6},  
\end{eqnarray*}
 $$
  A_{4}(z)=  \frac{2048}{5}A_{4,0}(z) =  28672 - 60928z + 42432z^{2} - 10608z^{3} + 663z^{4},
 $$
  $$
  B_{4}(z)=  \frac{2048}{5}B_{4,0}(z) =  28672 - 53760z + 31680z^{2} -6160z^{3} +  231z^{4},
  $$
  \begin{eqnarray*}
  F_{4}(z) =  &13989396348928 - 55957585395712z + 91916125077504z^{2} \\
  & - 79896826347520z^{3}+39463764078592z^{4} -11050000539648z^{5} \\
    & + 1648475542656z^{6}- 113348764800z^{7} + 2847396321z^{8}, 
\end{eqnarray*}
  $$
  A_{5}(z) = \frac{8192}{21}A_{5,0}(z) = 98304 - 258048z + 243712z^{2} - 99008z^{3} + 15912z^{4} - 663z^{5},
  $$
  $$
  B_{5}(z) = \frac{8192}{21}B_{5,0}(z) = 98304 - 233472z + 194560z^{2} - 66880z^{3} + 8360z^{4} -209z^{5}.
  $$
  and
\begin{eqnarray*}
 F_{5}(z) =  & 121733331812352  - 608666659061760z + 1301756554248192z^{2} \\
 & -1555026262622208z^{3} +1136607561252864z^{4}   - 523630732640256z^{5} \\
 & + 151029162176512z^{6} -26204424888320z^{7} +2515441608384z^{8} \\
  &  - 113971885760z^{9} +  1908029761z^{10}.
 \end{eqnarray*}
 We also define $A_{r}^{*}$ and $B_{r}^{*}$ via
 $$
 A_{r}^{*}(x , y) = x^{r}A_{r}(y/x) \; \mbox{ and } \;  B_{r}^{*}(x , y) = x^{r}B_{r}(y/x).
 $$
 Since $\Sigma_{r , 0} $ is assumed to be zero,
 $$
 \frac{\eta_{2}^{4}}{\xi_{2}^{4}} = \frac{\eta_{1}^{4} (B_{r}^{*}(\xi_{1}^{4}, \xi_{1}^{4} -\eta_{1}^{4}))^{4}} {\xi_{1}^{4} (A_{r}^{*}(\xi_{1}^{4} ,  \xi_{1}^{4} - \eta_{1}^{4}))^{4}} .
 $$
 Let $\mathfrak{I}_{r}$ be the integral ideal in $\mathbb{Q}(\sqrt{-t})$ generated by  $\xi_{1}^{4} (A_{r}^{*}(\xi_{1}^{4} ,  \xi_{1}^{4} - \eta_{1}^{4}))^{4}$ and $\eta_{1}^{4} (B_{r}^{*}(\xi_{1}^{4} ,  \xi_{1}^{4} - \eta_{1}^{4}))^{4}$, and $N(\mathfrak{I}_{r})$ be the absolute norm of $\mathfrak{I}_{r}$. Since the ideal generated by $\xi_{1}^{4} (A_{r}^{*}(\xi_{1}^{4} ,  \xi_{1}^{4} - \eta_{1}^{4}))^{4} - 
\eta_{1}^{4}( B_{r}^{*}(\xi_{1}^{4} ,  \xi_{1}^{4} - \eta_{1}^{4}))^{4}$ divides  $(\xi_{2}^{4} -\eta_{2}^{4})  \mathfrak{I}_{r}$, we obtain 
$$
|\xi_{1}|^{4(4r+1)} \left| A_{r}^{4}(z_{1}) - (1 - z_{1}) B_{r}^{4}(z_{1})\right|  =   \left|\xi_{1}^{4} (A_{r}^{*}(\xi_{1}^{4} ,  \xi_{1}^{4} - \eta_{1}^{4}))^{4} - \eta_{1}^{4} (B_{r}^{*}(\xi_{1}^{4} ,  \xi_{1}^{4} - \eta_{1}^{4}))^{4}\right|.
$$
From the fact that  $\mathfrak{I}_{r}$ is an imaginary quadratic field,
$$
|\xi_{1}|^{4(4r+1)} | A_{r}^{4}(z_{1}) - (1 - z_{1}) B_{r}^{4}(z_{1})| \leq  N(\mathfrak{I}_{r})^{1/2} |\xi_{2}^{4} - \eta_{2}^{4}| . 
$$
By  (\ref{ABF}),
$$
A_{r}^{4}(z_{1}) -  (1 - z_{1}) B_{r}^{4}(z_{1}) =  z_{1}^{2r+1} F_{r}(z_{1}),
$$
so we may conclude
$$
 |z_{1}|^{2r+1} |F_{r}(z_{1})| \leq  N(\mathfrak{I}_{r})^{1/2} |\xi_{2}^{4} - \eta_{2}^{4}| |\xi_{1}|^{-4(4r+1)} ;
 $$
i.e.
 $$
 1 \leq  \frac{ N(\mathfrak{I}_{r})^{1/2} |\xi_{2}^{4} - \eta_{2}^{4}| |\xi_{1}|^{-4(4r+1)}}{ |z_{1}|^{2r+1} |F_{r}(z_{1})| }.
 $$
Since $|z_{1}| = |\xi_{1}^{-4}||\xi_{1}^{4} - \eta_{1}^{4}|$ and $|\xi_{i}^{4} - \eta_{i}^{4}| = 8 P(x , y)$, it follows after a little work that
\begin{equation}\label{44}
|\xi_{1}|^{8r}  \leq  N(\mathfrak{I}_{r})^{1/2} \left| \xi_{1}^{4} - \eta_{1}^{4} \right|^{-4r-1} \left( 8P(x,y) \right)^{2r+1} \left| F_{r}(z_{1}) \right|^{-1}.
\end{equation}

To estimate  $N(\mathfrak{I}_{r})^{1/2} $, we choose a finite extension $\mathbf{M}$ of $\mathbb{Q}(\sqrt{-t})$ so that the ideal generated by $\xi_{1}^{4}$ and $\xi_{1}^{4} - \eta_{1}^{4}$ in $\mathbf{M}$ is a principal ideal, with generator $p$, say. We denote the extension of $\mathfrak{I}_{r}$ to $\mathbf{M}$, by $\mathfrak{I}'_{r}$. Let $\mathfrak{r}_{r}$ be the ideal in $\mathbf{M}$ generated by $A_{r}^{*}(u , v)$ and $B_{r}^{*}(u, v)$, where $u = \frac{\xi_{1}^{4}}{p}$ and  $v = \frac{\xi_{1}^{4} - \eta_{1}^{4}}{p}$. 
Since $A_{r}^{*}(x , x- y) = B_{r}^{*}(y , y-x)$,
\begin{eqnarray}\label{ss}
p^{4r+1}\mathfrak{r}_{r}^{4} B_{r}^{*}(0 , 1)^{4} &\subset &  p^{4r+1}\mathfrak{r}_{r}^{4} (u , B_{r}^{*}(0 , v)^{4}) ( u - v, B_{r}^{*}(0 , v)^{4})  \\ \nonumber
 &\subset &   p^{4r+1}\mathfrak{r}_{r}^{4} (u , B_{r}^{*}(0 , v)^{4}) ( u - v, A_{r}^{*}(v , v)^{4})  \\ \nonumber
 &\subset &   p^{4r+1}\mathfrak{r}_{r}^{4}( u , u - v) (u , B_{r}^{*}(u , v)^{4}) ( u - v, A_{r}^{*}(u , v)^{4})  \\ \nonumber
 &\subset & p^{4r+1} (uA^{*}(u , v)^{4} , (u - v)B_{r}^{*}(u , v)^{4}) = \mathfrak{I}'_{r},
 \end{eqnarray}
where $(m_{1} , \ldots , m_{n})$  denote the ideal in $\mathbf{M}$ generated by $m_{1} , \ldots , m_{n}$. 

We have
$$
A_{1}^{*}( x , y) - B_{1}^{*}(x , y) = -2y. 
$$
Therefore, 
$$
2(v) \subset (A_{1}^{*}( u , v) , B_{1}^{*}(u , v)) \subset  \mathfrak{r}_{1}, 
$$ 
where $(v)$ is the ideal generated by $v$ in $\mathbf{M}$.
Since $B_{1}^{*}(0 , 1) = -3$, it follows from (\ref{ss}) that 
$$
1296(\xi_{1}^{4} - \eta_{1}^{4})^{5} \subset 1296 p (\xi_{1}^{4} - \eta_{1}^{4})^{4}  = p^{5}16v^{4}B_{1}^{*}(0 , 1)^{4} \subset \mathfrak{I}'_{1}. 
$$
For $r = 2$, we first observe that 
$$
B_{1}^{*}(x , y)A_{2}^{*}( x , y) -A_{1}^{*}(x , y) B_{2}^{*}(x , y) =-10y^{3}
$$
and 
$$
(-32x + 7y)A_{2}^{*}( x , y) -(-32x+15y) B_{2}^{*}(x , y) = 80xy^{2}.
$$
Therefore, by (\ref{ss}) we have
$$
80(v)^{2}  \subset (-10v^{3} , 80uv^{2})  \subset (A_{2}^{*}( u , v) , B_{2}^{*}(u , v)) \subset  \mathfrak{r}_{2} .
$$ 
Since $B_{2}^{*}(0 , 1) =7$, we have 
$$  
80^{4}\times 7^{4} (\xi_{1}^{4} - \eta_{1}^{4})^{9}  \subset  80^{4}\times 7^{4} p (\xi_{1}^{4} - \eta_{1}^{4})^{8}  = 80^{4}p^{9}v^{8}  B_{2}^{*}(0 , 1)^{4} \subset \mathfrak{I}'_{2}.
$$
When $r = 3$, we have 
$$
B_{2}^{*}(x , y)A_{3}^{*}( x , y) -A_{2}^{*}(x , y) B_{3}^{*}(x , y) =-210y^{5}
$$
$$
(1616x^{2}-1078xy+77y^{2})A_{3}^{*}(x , y) - (1616x^{2}-1482xy+195y^{2}) B_{3}^{*}(x,y)= -16800x^{2}y^{3}.
$$
Substituting $77$ for $B_{3}^{*}(0 , 1)$, we conclude  
$$  
16800^{4}\times 77^{4} (\xi_{1}^{4} - \eta_{1}^{4})^{13}  \subset  16800^{4}\times 77^{4} p (\xi_{1}^{4} - \eta_{1}^{4})^{12}  = 16800^{4}p^{13}v^{12}  B_{3}^{*}(0 , 1)^{4} \subset \mathfrak{I}'_{3}.
$$
For $r = 4$, setting 
\begin{eqnarray*}
G_{4}(x , y)=14178304x^{3}-15889280x^{2}y+4071760xy^{2}-162393y^{3}, \\
H_{4}(x , y) = 14178304x^{3}-19433856x^{2}y+6714864xy^{2}- 466089y^{3},
 \end{eqnarray*}
we may verify that 
$$
B_{3}^{*}(x , y)A_{4}^{*}( x , y) -A_{3}^{*}(x , y) B_{4}^{*}(x , y) =-6006y^{7}
$$
and 
$$
G_{4}(x , y)A_{4}^{*}(x , y) - H_{4}(x , y)B_{4}^{*}(x , y) = -150678528y^{4}x^{3}.
$$
This implies that
$$
150678528 ^{4}\times 231^{4} (\xi_{1}^{4} - \eta_{1}^{4})^{17}  \subset  150678528^{4}\times 231^{4} p (\xi_{1}^{4} - \eta_{1}^{4})^{16}.
$$
Since this latter quantity is equal to $150678528^{4}p^{17}v^{16}  B_{4}^{*}(0 , 1)^{4}$, it follows that
$$
150678528 ^{4}\times 231^{4} (\xi_{1}^{4} - \eta_{1}^{4})^{17}  \subset \mathfrak{I}'_{4}.
$$
Finally, for $r = 5$, we have 
$$
B_{4}^{*}(x , y)A_{5}^{*}( x , y) -A_{4}^{*}(x , y) B_{5}^{*}(x , y) =-14586y^{7}
$$
and 
$$
G_{5}(x , y)A_{5}^{*}(x , y) - H_{5}(x , y)B_{5}^{*}(x , y) = - 134424576 y^{5}x^{4},
$$
where
\begin{eqnarray*}
G_{5}(x , y) = 43706368x^{4}-69346048x^{3}y+32767856x^{2}y^{2}-4764782{x}y^{3}+123519y^{4},\\ 
H_{5}(x , y)=43706368x^{4}-80272640x^{3}y+46006896x^{2}y^{2}-8845746xy^{3}+391833y^{4} .                                   
 \end{eqnarray*}  
 This implies that
$$
134424576^{4}\times 209^{4} (\xi_{1}^{4} - \eta_{1}^{4})^{21}  \subset 134424576^{4}\times 209^{4} p (\xi_{1}^{4} - \eta_{1}^{4})^{20} 
$$
whereby
$$
134424576^{4}\times 209^{4} (\xi_{1}^{4} - \eta_{1}^{4})^{21}  \subset  134424576 ^{4}p^{21}v^{20}  B_{5}^{*}(0 , 1)^{4} \subset \mathfrak{I}'_{5}.
$$

From the preceding arguments, we are thus able to deduce the following series of inequalities :
$$
N(\mathfrak{I}_{1})^{1/2} |\xi_{1}^{4} - \eta_{1}^{4}|^{-5} \leq 1296,
$$
$$
N(\mathfrak{I}_{2})^{1/2} |\xi_{1}^{4} - \eta_{1}^{4}|^{-9} \leq 560^{4},
$$
$$
N(\mathfrak{I}_{3})^{1/2} |\xi_{1}^{4} - \eta_{1}^{4}|^{-13} \leq (77\times16800)^{4},
$$
$$
N(\mathfrak{I}_{4})^{1/2} |\xi_{1}^{4} - \eta_{1}^{4}|^{-17} \leq (231\times150678528)^{4}
$$
and
 $$
 N(\mathfrak{I}_{5})^{1/2} |\xi_{1}^{4} - \eta_{1}^{4}|^{-21} \leq (209 \times 134424576)^{4}.
 $$
 These will enable us to contradict inequality (\ref{44}) for $r \leq 5$, provided we can find a suitably strong lower bound for $|\xi_{1}|$. Since $\xi_i^{4} = 4 (\sqrt{-t} + 1) (x_i - \sqrt{-t}y_i)^{4}$ and $x_1y_1 > 64 t^3$, via calculus we have that
\begin{equation} \label{raccoon}
|\xi_{1}|^{4} > 2^{16} \, t^{15/2},
\end{equation}
whence (\ref{44}) and the assumption that $P(x,y) \leq t^2$ imply 
\begin{equation} \label{idol}
2^{26r-3} t^{11r-2}  <  N(\mathfrak{I}_{r})^{1/2} \left| \xi_{1}^{4} - \eta_{1}^{4} \right|^{-4r-1} \left| F_{r}(z_{1}) \right|^{-1}.
\end{equation}
From (\ref{raccoon}), we have
$$
|z_{1}| = \left|\frac{8P (x , y)}{\xi_{1}^{4}}\right| < \left( 2^{13} \, t^{11/2} \right)^{-1} < 0.001,
$$
and consequently,
$$
F_{1}(z_{1}) > 10^{2} , \   F_{2}(z_{1}) > 10^{4}, \   F_{3}(z_{1}) > 10^{10}, \   F_{4}(z_{1}) > 10^{13}  \ \mbox{ and } \ F_{5}(z_{1}) > 10^{14}.
$$
In case $r=1$, inequality (\ref{idol}) thus implies that
$$
2^{23} t^9 < 6635.52 \times t^6,
$$
a contradiction for all $t$. Arguing similarly for $r=2, 3, 4$ and $5$, and noting that $t > 204$, completes the proof of Lemma \ref{mush}.
\end{proof}
         
\section{The Proof of Theorem \ref{main}}

Assume that there are two distinct coprime solutions $(x_{1} , y_{1})$ and $(x_{2} , y_{2})$ to inequality (\ref{2.1}) with $|\xi_{2}| > |\xi_{1}|$. We will show that $|\xi_{2}|$ is arbitrary large in relation to $|\xi_{1}|$. In particular, we will demonstrate via induction that
\begin{equation} \label{goat}
 |\xi_{2}| >\frac{\sqrt{r}}{5 \, t^{4r+7/4}} \left( \frac{4}{81} \right)^r \,  |\xi_{1}|^{4r + 3},
\end{equation}
for each positive integer $r$.
Since inequality (\ref{raccoon}) thus implies that
$$
 |\xi_{2}| >  t^{7r/2+31/8},
 $$
 for arbitrary $r$, we deduce an immediate contradiction.
 
 We first prove inequality (\ref{goat}) for $r=1$. By (\ref{Gap}) and (\ref{raccoon}),
  $$
  c_{1}(1 ,0) \, |\xi_{1}|^{5}|\xi_{2}|^{-3} < 2^{-13} \pi^{-1/2} t^{-5/2} < 0.1,
  $$
and hence,  since $\Sigma_{1,0} \neq 0$, Lemma \ref{c} yields
$$
c_{2}(1 , 0) |\xi_{1}|^{-7}|\xi_{2}| > 0.9,
$$
which, after a little work, implies (\ref{goat}).

We now proceed by induction. Suppose that (\ref{goat}) holds for some $r \geq 1$. Then
$$
c_{1}(r+1 ,0) |\xi_{1}|^{4r+5}|\xi_{2}|^{-3} < \frac{2000}{\sqrt{\pi} r^2} \, t^{12r+13/2} \, \left( \frac{3^{12}}{2^4} \right)^r \, |\xi_{1}|^{-8r-4},
$$
and hence, from (\ref{raccoon}),
$$
c_{1}(r+1 ,0) |\xi_{1}|^{4r+5}|\xi_{2}|^{-3} < \frac{125}{2^{12} \sqrt{\pi} r^2} \, t^{-3r-1} \, \left( \frac{3^{12}}{2^{36}} \right)^r < 0.1.
$$
If $\Sigma_{r+1 , 0} \neq 0$, then  by Lemma \ref{c},
$$ 
c_{2}(r+1 , 0) |\xi_{1}|^{-4(r+1)-3}|\xi_{2}| > 0.9,
$$
which leads to inequality (\ref{goat}) with $r$ replaced by $r+1$. If, however, $\Sigma_{r+1, 0} = 0$ then  by Lemmas \ref{nv} and \ref{mush}, both $\Sigma_{r+1,1}$ and $\Sigma_{r+2 ,1}$ are nonzero, and $r \geq 5$. Using the induction hypothesis, we find as previously that
$$
c_{1}(r + 1,1)  |\xi_{1}|^{4r+4}|\xi_{2}|^{-3} < 0.1
$$
and thus by Lemma \ref{c}  conclude that
$$
c_{2}(r+1,1) |\xi_{1}|^{-4r-4}|\xi_{2}| >0.9.
$$
It follows that
$$
|\xi_{2}| > 0.08 \times \frac{\sqrt{r+1}}{ t^{4r+29/8}} \, \left( \frac{4}{81} \right)^r \, |\xi_{1}|^{4r+4}.
$$
Consequently,
$$
c_{1}(r+2 , 1) |\xi_{1}|^{4r+8}|\xi_{2}|^{-3} < \frac{ 24000}{(r+1)^2} \, \left( \frac{3^{12}}{2^4} \right)^r \, t^{12r+25/2} \,  |\xi_{1}|^{-8r-4},
$$
whereby, from (\ref{raccoon}) and the fact that $r \geq 5$, 
$$
c_{1}(r+2 , 1) |\xi_{1}|^{4r+8}|\xi_{2}|^{-3} < \frac{1}{2 (r+1)^2} \, \left( \frac{3^{12}}{2^{36}} \right)^r \, t^{5-3r} < 0.1.
$$
 Lemma \ref{c} thus implies the inequality
$$
c_{2}(r+2,1) |\xi_{1}|^{-4r-8}|\xi_{2}| > 0.9
$$
and so
$$
|\xi_{2}| > 0.08 \, \sqrt{r+1} \, \left( \frac{4}{81} \right)^{r+1} t^{-4r-61/8} |\xi_{1}|^{4r+8}.
$$
From (\ref{raccoon}), it follows that
\begin{equation*}
|\xi_{2}| > \frac{\sqrt{r+1}}{5t^{4r+4+7/4}}\left( \frac{4}{81} \right)^{r+1} |\xi_{1}|^{4r+7},
\end{equation*}
as desired. This completes the proof of inequality (\ref{goat}) and hence we conclude that there is at most one solution to (\ref{2.1}) related to each fourth root of unity.

To finish the proof of Theorem \ref{main}, it is enough to show that three of the roots of unity under consideration do not have solutions to (\ref{2.1}) associated to them. We recall that the polynomial
$$
P(x, 1) =  x^{4} + 4tx^{3} -6tx^{2} - 4t^{2}x + t^{2}
$$
has $4$ real roots $\beta_{1}$, $\beta_{2}$, $\beta_{3}$, $\beta_{4}$, say, where
\begin{eqnarray}\nonumber
\sqrt{t}+\frac{1}{2}+\frac{1}{8\sqrt{t} }- \frac{2}{8t}   <   \beta_{1} <  \sqrt{t}+\frac{1}{2}+\frac{1}{8\sqrt{t}} - \frac{1}{8t} \\ \nonumber
-\sqrt{t}+\frac{1}{2}-\frac{1}{8\sqrt{t}}- \frac{1}{8t}  <  \beta_{2}  <  -\sqrt{t}+\frac{1}{2}-\frac{1}{8\sqrt{t}} \\ \nonumber
\frac{1}{4} - \frac{5}{64t} + \frac{22}{512t^{2}}  <  \beta_{3}  <   \frac{1}{4} - \frac{5}{64t} + \frac{23}{512t^{2}} \\ \nonumber
-4t -\frac{5}{4}+\frac{21}{64t}-\frac{87}{512t^{2}}  <  \beta_{4}  < -4t -\frac{5}{4}+\frac{21}{64t}-\frac{84}{512t^{2}} .
\end{eqnarray}
(since $t \geq 18$, the polynomial $P(x , 1)$ changes sign between the given bounds).
Since
$$
P(\beta_{i} , 1) = \frac{1}{8} ( \xi^{4}(\beta_{i} , 1) - \eta^{4}(\beta_{i} , 1) ) = 0,
$$
it follows that, for each $1 \leq i \leq 4$,
$\frac{\eta(\beta_{i} , 1)}{\xi(\beta_{i} , 1)}$ is a fourth root of unity.  Noting that
$$
\frac{\eta(\beta_{i} , 1)}{\xi(\beta_{i} , 1)} - \frac{\eta(\beta_{j} , 1)}{\xi(\beta_{j} , 1)} = \left(\frac{\sqrt{-t}-1}{\sqrt{-t}+1} \right)^{1/4} \frac{2\sqrt{-t} (\beta_{j} - \beta_{i})}{(\beta_{i}-\sqrt{-t}) (\beta_{j} - \sqrt{-t})},
$$
they are in fact distinct. We now proceed to show that solutions to (\ref{2.1}) necessarily correspond to fourth roots of unity related to $\beta_2$.

In  \cite{wal}, it is shown that for $\{V_{2n+1}\}$ defined in Section \ref{sec2}, the equation $z^{2} = V_{4n+1}$ has no solution.  Supposing that $z^{2} = V_{4n+3}$, as  in the proof of Proposition \ref{tvw}, there exist integers $t_{1}$, $t_{2}$, $G$ and $H$,  so that the integers $x, y$ arising from Proposition \ref{tvw} satisfy $x=-t_{1}G$ and $y = H$. We have
\begin{eqnarray*}
\frac{x}{y} = \frac{-t_{1}G}{H} = \frac{-2t_{1}G^{2}}{2GH} & = & -\frac{\sqrt{V_{4n+3}} - V_{2n+1}}{U_{n+1}} \\
&= &-\frac{\sqrt{V_{2n+2}^{2}+V_{2n+1}^{2}} - V_{2n+1}}{\frac{V_{2n+2}}{\sqrt{t}}}\\
&=& -\sqrt{t} \left( \sqrt{1 + \frac{V_{2n+1}^{2}}{V_{2n+2}^{2}}} - \frac{V_{2n+1}}{V_{2n+2}} \right),\\
\end{eqnarray*}
using the fact that $V_{2n+2}= \sqrt{t} \, U_{n+1}$. Thus
$$
\left| \frac{x}{y}+\sqrt{t} \right| = \sqrt{t} \,  \left| \sqrt{1 + \frac{V_{2n+1}^{2}}{V_{2n+2}^{2}}} - \frac{V_{2n+1}}{V_{2n+2}}-1 \right|.
$$
A crude application of the Mean Value Theorem therefore implies that
$$
\left| \frac{x}{y}+\sqrt{t} \right| < \sqrt{t}
$$
and consequently, $x/y \in (-2 \sqrt{t}, 0)$, whereby the inequalities for $\beta_i$ yield
$$
\left| \frac{x}{y} - \beta_{1} \right| \geq  \left| \sqrt{t} + \beta_{1} \right| -   \left| \frac{x}{y} + \sqrt{t} \right| > 2\sqrt{t} -\sqrt{t} = \sqrt{t},
$$
$$
\left| \frac{x}{y} - \beta_{3} \right| \geq  \left| \sqrt{t} + \beta_{3} \right| -   \left| \frac{x}{y} + \sqrt{t} \right| > \sqrt{t}+\frac{1}{5} - \sqrt{t} = \frac{1}{5}
$$ 
and 
$$
\left| \frac{x}{y} - \beta_{4} \right| \geq  \left| \sqrt{t} + \beta_{4} \right| -   \left| \frac{x}{y} + \sqrt{t} \right| > 3t - \sqrt{t} > 2t.
$$ 

Let $\beta \in \{\beta_1, \beta_{3} , \beta_{4}\}$. We have just shown that if $(x , y)$ is a solution to inequality (\ref{2.1}), then
$$
\left|\frac{x}{y} - \beta \right| > \frac{1}{5}.
$$
If we suppose that $\omega = \frac{\eta(\beta,1)}{\xi(\beta , 1)}$, then
$$
\left|\omega - \frac{\xi(x , y)}{\eta(x , y)} \right| = \left|\frac{\eta(\beta , 1)}{\xi(\beta , 1)} - \frac{\eta(\frac{x}{y} , 1)}{\xi(\frac{x}{y} , 1)} \right| = \left|\frac{2\sqrt{-t} (\frac{x}{y} - \beta)}{(\beta-\sqrt{-t}) (\frac{x}{y} - \sqrt{-t})} \right|,
$$
whence the inequalities
$$
|\beta-\sqrt{-t}|< \sqrt{16t^{2}+17t}
$$
and 
$$
\left|\frac{x}{y} - \sqrt{-t} \right| < \sqrt{5 t} 
$$
(recall that $|x/y| < 2 \sqrt{t}$) imply
$$
\left|\omega - \frac{\xi(x , y)}{\eta(x , y)} \right| > \frac{2 }{5 \sqrt{80t^{2}+85t}}.
$$
Since 
$$
|z| = \frac{8P(x , y)}{|\xi^{4}(x , y)|} \leq \frac{8 t^2}{|\xi^{4}(x , y)|},
$$
this, together with (\ref{raccoon}), contradicts Lemma \ref{lipp}.

This shows that there is no solution related to three of the fourth roots of unity (those corresponding to $\beta_1, \beta_{3}$ and $\beta_{4}$). Therefore, there is at most a single solution to inequality (\ref{2.1}). Together with Propositions \ref{btw} and \ref{tvw}, this completes the proof of Theorem \ref{main}.

\section{Acknowledgments} 
The author would like to thank Professor Michael Bennett for encouragement, careful reading and various insightful comments and suggestions.


 \end{document}